\documentclass[a4paper]{article}
\usepackage{amsmath,amssymb}
\usepackage{amsthm}
\usepackage{enumerate}
\usepackage[final]{graphicx}
\usepackage{epstopdf}
\makeatletter
\@addtoreset{equation}{section}
\@addtoreset{table}{section}
\makeatother
\usepackage[text={28cc,42cc},centering]{geometry}

\newcommand\ol{\overline}
\newcommand\ga{\alpha}
\newcommand\tov[1]{\xrightarrow{v}{#1}}

\newcommand\cl[1][G]{\ensuremath{\mathrm{cl}(#1)}}

\newcommand\se{\subseteq}
\newcommand\beq[1]{\begin{equation}\label{#1}}
\newcommand\eeq{\end{equation}}

\newtheorem{theorem}{Theorem}[section]
\newtheorem{lem}{Lemma}[section]
\newtheorem{cor}{Corollary}[section]
\theoremstyle{definition}

\begin{document}

\author{Nedyalko Nenov}
\title{ON THE VERTEX FOLKMAN NUMBERS $F_v(\underbrace{2,\dots,2}_r;r-1)$
and $F_v(\underbrace{2,\dots,2}_r;r-2)$%
\thanks{This work is supported by the Scientific Research Fund of
the St. Kl.~Ohridski Sofia University under contracts~85/2006.}}

\maketitle

\begin{abstract}
For a graph $G$ the symbol $G\tov(a_1,\dots,a_r)$ means that in every
$r$-coloring of the vertices of $G$, for some $i\in\{1,2,\dots,r\}$,
there exists a monochromatic $a_i$-clique of color $i$. The vertex
Folkman numbers
\[
F_v(a_1,\dots,a_r;q)=\min\{|V(G)|:G\tov(a_1,\dots,a_r)
\text{ and }
K_q\nsubseteq G\}
\]
are considered. We prove that $F_v(\underbrace{2,\dots,2}_r;r-1)=r+7$,
$r\ge 6$ and $F_v(\underbrace{2,\dots,2}_r;r-2)=r+9$, $r\ge 8$.

\textbf{Keywords.} Folkman graphs, Folkman numbers

\textbf{2000 Math.~Subject Classification.} 05C55
\end{abstract}

\section{Introduction}

We consider only finite, non-oriented graphs without loops and multiple edges.
We call a $p$-clique of the graph $G$ a set of $p$ vertices, each two of which
are adjacent. The largest positive integer $p$, such that the graph $G$
contains a $p$-clique is denoted by \cl. In this paper we shall also use
the following notations:
\begin{itemize}
  \item $V(G)$ is the vertex set of the graph $G$;
  \item $E(G)$ is the edge set of the graph $G$;
  \item $\ol G$ is the complement of $G$;
  \item $G[V]$, $V\se V(G)$ is the subgraph of $G$ induced by $V$;
  \item $G-V$,  $V\se V(G)$ is the subgraph of $G$ induced by $V(G)\setminus V$;
  \item $\ga(G)$ is the vertex independence number of $G$;
  \item $\chi(G)$ is the chromatic number of $G$;
  \item $f(G)=\chi(G)-\cl$;
  \item $K_n$ is the complete graph on $n$ vertices;
  \item $C_n$ is the simple cycle on $n$ vertices.
\end{itemize}

Let $G_1$ and $G_2$ be two graphs without common vertices.
We denote by $G_1+G_2$ the graph $G$ for which $V(G)=V(G_1)\cup V(G_2)$ and
$E(G)=E(G_1)\cup E(G_2)\cup E'$, where $E'=\{[x,y]:x\in V(G_1), y\in V(G_2)\}$.

The Ramsey number $R(p,q)$ is the smallest natural $n$ such that for every
$n$-vertex graph $G$ either $\cl\ge p$ or $\ga(G)\ge q$. An exposition of
the results on the Ramsey numbers is given in \cite{NN:25}. We shall need
Table~\ref{NN:t1} of the known Ramsey numbers $R(p,3)$ (see \cite{NN:25}).
\begin{table}[h]
\centering
\begin{tabular}{|l|*8{c|}}
\hline
$p$&3&4&5&6&7&8&9&10\\
\hline
$R(p,3)$&6&9&14&18&23&28&36&40--43\\
\hline
\end{tabular}
\caption{Known Ramsey numbers}\label{NN:t1}
\end{table}

\textbf{Definition.}
Let $a_1,\dots,a_r$ be positive integers. We say that the $r$-coloring
\[
V(G)=V_1\cup\dots\cup V_r,
\quad
V_i\cap V_j=\emptyset,\ i\ne j
\]
of the vertices of the graph $G$ is $(a_1,\dots,a_r)$-free, if $V_i$
does not contain an $a_i$-clique for each $i\in\{1,\dots,r\}$.
The symbol $G\tov(a_1,\dots,a_r)$ means that there is no
$(a_1,\dots,a_r)$-free coloring of the vertices of $G$.

Let $a_1,\dots,a_r$ and $q$ be natural numbers. Define
\begin{align*}
H_v(a_1,\dots,a_r;q)&=\{G:G\tov(a_1,\dots,a_r)\text{ and }\cl<q\},\\
F_v(a_1,\dots,a_r;q)&=\min\{|V(G)|:G\in H_v(a_1,\dots,a_r;q)\}.
\end{align*}
We say that the graph $G\in H_v(a_1,\dots,a_r;q)$ is an extremal graph in
$H_v(a_1,\dots,a_r;q)$, if $|V(G)|=F_v(a_1,\dots,a_r;q)$.

It is clear that $G\tov(a_1,\dots,a_r)$ implies $\cl\ge\max\{a_1,\dots,a_r\}$.
Folkman~\cite{NN:3} proved that there exists a graph $G$ such that
$G\tov(a_1,\dots,a_r)$ and $\cl=\max\{a_1,\dots,a_r\}$. Therefore
\beq{NN:eq11}
F_v(a_1,\dots,a_r;q)\text{ exists}\iff q>\max\{a_1,\dots,a_r\}.
\eeq
The numbers $F_v(a_1,\dots,a_r;q)$ are called vertex Folkman numbers.

If $a_1,\dots,a_r$ are positive integers, $r\ge 2$ and $a_i=1$ then it is
easily to see that
\[
G\tov(a_1,\dots,a_i,\dots,a_r)\iff G\tov(a_1,\dots,a_{i-1},a_{i+1},a_r).
\]
Thus it is enough to consider just such numbers $F_v(a_1,\dots,a_r;q)$ for
which $a_i\ge 2$, $i=1,\dots,r$. In this paper we consider the vertex Folkman
numbers $F_v(2,\dots,2;q)$. Set
\[
(\underbrace{2,\dots,2}_r)=(2_r)
\quad\text{and}\quad
F_v(\underbrace{2,\dots,2}_r;q)=F_v(2_r;q).
\]
By~\eqref{NN:eq11}
\beq{NN:eq12}
F_v(2_r;q)\text{ exists}\iff q\ge 3.
\eeq
It is clear that
\beq{NN:eq13}
G\tov(2_r)\iff\chi(G)\ge r+1.
\eeq
Since $K_{r+1}\tov(2_r)$ and $K_r\not\tov(2_r)$ we have
\[
F_v(2_r;q)=r+1
\quad\text{if}\quad
q\ge r+2.
\]

In \cite{NN:2} Dirac proved the following

\begin{theorem}[\cite{NN:2}]\label{NN:th11}
Let $G$ be a graph such that $\chi(G)\ge r+1$ and $\cl\le r$. Then
\begin{enumerate}[\indent\rm(a)]
\item
$|V(G)|\ge r+3$;
\item
If $|V(G)|=r+3$ then $G=K_{r-3}+C_5$.
\end{enumerate}
\end{theorem}

According to~\eqref{NN:eq13}, Theorem~\ref{NN:th11} is equivalent to
the following

\begin{theorem}\label{NN:th12}
Let $r\ge 2$ be a positive integer. Then
\begin{enumerate}[\indent\rm(a)]
\item
$F_v(2_r;r+1)=r+3$;
\item
$K_{r-3}+C_5$ is the only extremal graph in $H_v(2_r;r+1)$.
\end{enumerate}
\end{theorem}

In \cite{NN:14} {\L}uczak, Ruci\'nski and Urba\'nski defined
for arbitrary positive integers $a_1,\dots,a_r$ the numbers
\beq{NN:eq14}
m=\sum_{i=1}^r(a_i-1)+1
\quad\text{and}\quad
p=\max\{a_1,\dots,a_r\}.
\eeq
%In \cite{NN:14} the authors
and proved the following extension of Theorem~\ref{NN:th12}.

\begin{theorem}[\cite{NN:14}]\label{NN:th13}
Let $a_1,\dots,a_r$ be positive integers and $m$ and $p$ be defined
by~\eqref{NN:eq14}. Let $m\ge p+1$. Then
\begin{enumerate}[\indent\rm(a)]
\item
$F_v(a_1,\dots,a_r;m)=m+p$;
\item
$K_{m-p-1}+\ol C_{2p+1}$ is the only extremal graph in $H_v(a_1,\dots,a_r;m)$.
\end{enumerate}
\end{theorem}

An another extension of Theorem~\ref{NN:th11} was given in \cite{NN:21}.

From~\eqref{NN:eq11} it follows that the numbers $F_v(a_1,\dots,a_r;m-1)$
exists if and only if $m\ge p+2$. The exact values of all numbers
$F_v(a_1,\dots,a_r;m-1)$ for which $p=\max\{a_1,\dots,a_r\}\le 4$ are known.
A detailed exposition of these results was given in \cite{NN:13}
and \cite{NN:23}. We do not know any exact values of $F_v(a_1,\dots,a_r;m-1)$
in the case when $\max\{a_1,\dots,a_r\}\ge 5$. Here we shall note only
the values $F_v(a_1,\dots,a_r;m-1)$ when $a_1=a_2=\dots=a_r=2$, i.e.
of the numbers $F_v(2_r;r)$. From~\eqref{NN:eq12} these numbers exist
if and only if $r\ge 3$. If $r=3$ and $r=4$ we have that
\begin{align}\label{NN:eq15}
F_v(2_3;3)&=11;\\
F_v(2_4;4)&=11.
\label{NN:eq16}
\end{align}

The inequality $F_v(2_3;3)\le 11$ was proved in \cite{NN:15} and the opposite
inequality $F_v(2_3;3)\ge 11$ was proved in \cite{NN:1}.
The equality~\eqref{NN:eq16} was proved in\cite{NN:18} (see also \cite{NN:19}).
If $r\ge 5$ we have the following

\begin{theorem}[\cite{NN:17}, see also \cite{NN:24}]\label{NN:th14}
Let $r\ge 5$. Then:
\begin{enumerate}[\indent\rm(a)]
\item
$F_v(2_r;r)=r+5$;
\item
$K_{r-5}+C_5+C_5$ is the only extremal graph in $H_v(2_r;r)$.
\end{enumerate}
\end{theorem}

Theorem~\ref{NN:th14}(a) was proved also in \cite{NN:8} and \cite{NN:14}.

According to \eqref{NN:eq12} the number $F_v(2_r;r-1)$ exists if and only if
$r\ge 4$. In \cite{NN:17} we prove that
\beq{NN:eq17}
F_v(2_r;r-1)=r+7
\text{ if }
r\ge 8.
\eeq

In this paper we shall improve~\eqref{NN:eq17} proving the following

\begin{theorem}\label{NN:th15}
Let $r\ge 4$ be an integer. Then
\begin{enumerate}[\indent\rm(a)]
\item
$F_v(2_r;r-1)\ge r+7$;
\item
$F_v(2_r;r-1)=r+7$ if $r\ge 6$;
\item
$F_v(2_5;4)\le 16$.
\end{enumerate}
\end{theorem}

In \cite{NN:9} Jensen and Royle show that
\beq{NN:eq18}
F_v(2_4;3)=22.
\eeq

We see from Theorem~\ref{NN:th15} and~\eqref{NN:eq18} that $F_v(2_5;4)$
is the only unknown number of the kind $F(2_r;r-1)$.

From~\eqref{NN:eq12} the Folkman number $F(2_r;r-2)$ exists if and only if
$r\ge 5$. In \cite{NN:16} we prove that $F_v(2_r;r-2)=r+9$ if $r\ge 11$.
In this paper we shall improve this result proving the folloing

\begin{theorem}\label{NN:th16}
Let $r\ge 5$ be an integer. Then
\begin{enumerate}[\indent\rm(a)]
\item
$F_v(2_r;r-2)\ge r+9$;
\item
$F_v(2_r;r-2)=r+9$ if $r\ge 8$.
\end{enumerate}
\end{theorem}

The numbers $F_v(2_r;r-2)$, $5\le r\le 7$, are unknown.

\section{Auxiliary results}

Let $G$ be an arbitrary graph. Define
\[
f(G)=\chi(G)-\cl.
\]

\begin{lem}\label{NN:l21}
Let $G$ be a graph such that $f(G)\le 2$. Then $|V(G)|\ge\chi(G)+2f(G)$.
\end{lem}

\begin{proof}[\bf Proof.]
Since $\chi(G)\ge\cl$, $f(G)\ge 0$. If $f(G)=0$ the inequality is trivial.
Let $f(G)=1$ and $\chi(G)=r+1$. Then $\cl=r$. Note that $r\ge 2$ because
$\chi(G)\ne\cl$. By~\eqref{NN:eq13}, $G\in H_v(2_r;r+1)$. Thus, from
Theorem~\ref{NN:th12}(a), it follows $|V(G)|\ge r+3=2f(G)+\chi(G)$.
Let $f(G)=2$ and $\chi(G)=r+1$. Then $\cl=r-1$. Since $\chi(G)\ne\cl$,
$\cl=r-1\ge 2$, i.e. $r\ge 3$. From Theorem~\ref{NN:th14}(a), \eqref{NN:eq15}
and \eqref{NN:eq16} we obtain that $|V(G)|\ge r+5=\chi(G)+2f(G)$. This
completes the proof of Lemma~\ref{NN:l21}.
\end{proof}

Let $G=G_1+G_2$. Obviously,
\begin{align}\label{NN:eq21}
\chi(G)&=\chi(G_1)+\chi(G_2);\\
\cl&=\cl[G_1]+\cl[G_2].\label{NN:eq22}
\end{align}

Hence,
\beq{NN:eq23}
f(G)=f(G_1)+f(G_2).
\eeq

A graph $G$ is defined to be vertex-critical chromatic if $\chi(G-v)<\chi(G)$
for all $v\in V(G)$. We shall use the following result in the proof of
Theorem~\ref{NN:th16}.

\begin{theorem}[\cite{NN:4}, see also \cite{NN:5}]\label{NN:th21}
Let $G$ be a vertex-critical chromatic graph and $\chi(G)\ge 2$.
If $|V(G)|<2\chi(G)-1$, then $G=G_1+G_2$, where $V(G_i)\ne\emptyset$,
$i=1,2$.
\end{theorem}

\textbf{Remark.}
In the original statement of Theorem~\ref{NN:th21} the graph $G$ is
edge-critical chromatic (and not vertex-critical chromatic). Since
each vertex-critical chromatic graph $G$ contains an edge-critical
chromatic subgraph $H$ such that $\chi(G)=\chi(H)$ and $V(G)=V(H)$,
the above statement of this theorem is equivalent to the original one.
They are also more convenient for the proof of Theorem~\ref{NN:th16}.

Let $G$ be a graph and $A\subseteq V(G)$ be an independent set of vertices of
the graph $G$. It is easy to see that
\beq{NN:eq24}
G\tov(2_r),\ r\ge 2
\Rightarrow
G-A\tov(2_{r-1}).
\eeq

\begin{lem}\label{NN:l22}
Let $G\in H_v(2_r;q)$, $q\ge 3$ and $|V(G)|=F_v(2_r;q)$. Then
\begin{enumerate}[\indent\rm(a)]
\item
$G$ is a vertex-critical $(r+1)$-chromatic graph;
\item
if $q<r+3$ then $\cl=q-1$.
\end{enumerate}
\end{lem}

\begin{proof}[\bf Proof.]
By~\eqref{NN:eq13}, $\chi(G)\ge r+1$. Assume that (a) is wrong. Then there
exists $v\in V(G)$ such that $\chi(G-v)\ge r+1$. According to~\eqref{NN:eq13},
$G-v\in H_v(2_r;q)$. This contradicts the equality $|V(G)|=F_v(2_r;q)$.

Assume that (b) is wrong, i.e. $\cl\le q-2$. Then from $q<r+3$ it follows
that $\cl<r+1$. Since $\chi(G)\ge r+1$ there are $a,b\in V(G)$ such that
$[a,b]\notin E(G)$. Consider the subgraph $G_1=G-\{a,b\}$. We have $r\ge 2$,
because $\chi(G)\ne\cl$. Thus, from~\eqref{NN:eq24} and $\cl\le q-2$ it follows
that $G_1\in H_v(2_{r-1};q-1)$. Obviously, $G_1\in H_v(2_{r-1};q-1)$ leads to
$K_1+G_1\in H_v(2_r;q)$. This contradicts the equality $|V(G)|=F_v(2_r;q)$,
because $|V(K_1+G_1)|=|V(G)|-1$. Lemma~\ref{NN:l22} is proved.
\end{proof}

\begin{lem}\label{NN:l23}
Let $G\in H_v(2_r;q)$, $r\ge 2$. Then
\[
|V(G)|\ge F_v(2_{r-1};q)+\ga(G).
\]
\end{lem}

\begin{proof}[\bf Proof.]
Let $A\se V(G)$ be an independent set such that $|A|=\ga(G)$. Consider
the subgraph $G_1=G-A$. According to~\eqref{NN:eq24}, $G_1\in H_v(2_{r-1};q)$.
Hence $|V(G_1)|\ge F_v(2_{r-1};q)$. Since $|V(G)|=|V(G_1)|+\ga(G)$,
Lemma~\ref{NN:l23} is proved.
\end{proof}

We shall use also the following three results:
\begin{alignat}{2}\label{NN:eq25}
&F_v(2,2,p;p+1)\ge 2p+4,&&\qquad\text{see \cite{NN:20};}\\
&F_v(2,2,4;5)=13,&&\qquad\text{see \cite{NN:22}.}\label{NN:eq26}
\end{alignat}
 \begin{theorem}[\cite{NN:12}]\label{NN:th22}
Let $G$ be a graph, $\cl\le p$ and $|V(G)|\ge p+2$, $p\ge 2$.
Let $G$ also have the following two properties:
\begin{enumerate}[\indent\rm(i)]
\item
$G\not\tov(2,2,p)$;
\item
If $V(G)=V_1\cup V_2\cup V_3$ is a $(2,2,p)$-free 3-coloring then
$|V_1|+|V_2|\le 3$.
\end{enumerate}

Then $G=K_1+G_1$.
\end{theorem}

\section{An upper bound for the numbers $F_v(2_r;q)$}

Consider the graph $P$, whose complementary graph $\ol P$ is given in
Fig.~\ref{NN:f1}.
\begin{figure}[b]
\centering
\includegraphics[scale=1.5]{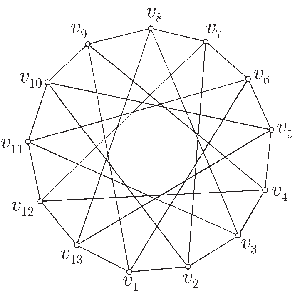}
\caption{Graph $\ol P$}\label{NN:f1}
\end{figure}
This graph is a well-known construction of Greenwood and Gleason~\cite{NN:6},
which shows that $R(5,3)\ge 14$, because $|V(P)|=13$ and
\begin{align}\label{NN:eq31}
\ga(P)&=2;\\
\cl[P]&=4
\qquad
\text{see \cite{NN:6}.}\label{NN:eq32}
\end{align}
From $|V(P)|=13$ and~\eqref{NN:eq31} it follows that $\chi(P)\ge 7$. Since
\(
\{v_1\}\cup\{v_2,v_3\}\cup\dots\cup\{v_{12},v_{13}\}
\)
is a 7-chromatic partition of $V(P)$, we have
\beq{NN:eq33}
\chi(P)=7.
\eeq

Let $r$ and $s$ be non-negative integers and $r\ge 3s+6$. Define
\[
\tilde P=K_{r-3s-6}+P+\underbrace{C_5+\dots+C_5}_s.
\]
From~\eqref{NN:eq21}, \eqref{NN:eq22}, \eqref{NN:eq32} and~\eqref{NN:eq33}
we obtain that $\chi(\tilde P)=r+1$ and $\cl[\tilde P]=r-s-2$.
By~\eqref{NN:eq13} $\tilde P\in H_v(2_r;r-s-1)$ and thus
\[
F_v(2_r;r-s-1)\le|V(\tilde P)|.
\]
Since $|V(\tilde P)|=r+2s+7$, we prove the following

\begin{theorem}\label{NN:th31}
Let $r$ and $s$ be non-negative integers and $r\ge 3s+6$. Then
\[
F_v(2_r;r-s-1)\le r+2s+7.
\]
\end{theorem}

\textbf{Remark.}
Since $r\ge 3s+6$ we have $r-s-1>2$. Thus, according to (1.2), the numbers
$F_v(2_r;r-s-1)$ exists.

\section{Proof of Theorem~\ref{NN:th15}}

\begin{proof}[Proof of Theorem~\ref{NN:th15}(a)]
Let $G\in H_v(2_r;r-1)$. We need to prove that $|V(G)|\ge r+7$.
From Lemma~\ref{NN:l23} we have that
\[
|V(G)|\ge F_v(2_{r-1};r-1)+\ga(G).
\]
By~\eqref{NN:eq15}, \eqref{NN:eq16} and Theorem~\ref{NN:th14}(a),
$F_v(2_{r-1};r-1)\ge r+4$. Hence
\beq{NN:eq41}
|V(G)|\ge r+4+\ga(G).
\eeq
We prove the inequality $|V(G)|\ge r+7$ by induction on $r$.
From Table~\ref{NN:t1} we see that
\beq{NN:eq42}
R(r-1,3)<r+6
\text{ if $r=4$ or $r=5$.}
\eeq
Obviously, from $G\in H_v(2_r;r-1)$ it follows that $\chi(G)\ne\cl$. Thus,
$\ga(G)\ge 2$. From~\eqref{NN:eq41} we obtain $|V(G)|\ge r+6$. From this
inequality and~\eqref{NN:eq42} we see that $|V(G)|>R(r-1,3)$ if $r=4$ or
$r=5$. Since $\cl<r-1$, it follows that $\ga(G)\ge 3$.
Now from~\eqref{NN:eq41} we obtain that $|V(G)|\ge r+7$ if $r=4$ or $r=5$.

Let $r\ge 6$. We shall consider two cases:

\textbf{Case 1.}
$G\not\tov(2,2,r-2)$. From Theorem~\ref{NN:th22} we see that only following
two subcases are possible:

\textbf{Subcase 1a.}
$G=K_1+G_1$. From $G\in H_v(2_r,r-1)$ it follows that
$G_1\in H_v(2_{r-1};r-2)$. By the inductive hypothesis, $|V(G_1)|\ge r+6$.
Therefore, $|V(G)|\ge r+7$.

\textbf{Subcase 1b.}
There is a $(2,2,r-2)$-free 3-coloring $V(G)=V_1\cup V_2\cup V_3$ such that
$|V_1|+|V_2|\ge 4$. Consider the subgraph $\tilde G=G[V_3]$. By assumption
$\tilde G$ does not contain an $(r-2)$-clique, i.e. $\cl[\tilde G]<r-2$.
Since $V_1$ and $V_2$ are independent sets and $G\tov(2_r)$ it follows
from~\eqref{NN:eq24} that $\tilde G\tov(2_{r-2})$. Thus,
$\tilde G\in H_v(2_{r-2};r-2)$. By~\eqref{NN:eq16} and Theorem~\ref{NN:th14}(a),
$|V(\tilde G)|\ge r+3$. As $|V_1|+|V_2|\ge 4$, we have $|V(G)|\ge 7$.

\textbf{Case 2.}
$G\tov(2,2,r-2)$. Since $\cl<r-1$, $G\in H_v(2,2,r-2;r-1)$.
From~\eqref{NN:eq25} it follows that $|V(G)|\ge 2(r-2)+4=2r$.
Hence, if $2r\ge r+7$, i.e. $r\ge 7$, we have $|V(G)|\ge r+7$.
Let $r=6$. Then $G\in H_v(2,2,4;5)$. By~\eqref{NN:eq26}, $|V(G)|\ge 13$.
\end{proof}

\begin{proof}[Proof of Theorem~\ref{NN:th15}(b)]
Let $r\ge 6$. According to Theorem~\ref{NN:th15}(a) we have
$F_v(2_r;\allowbreak r-1)\ge r+7$. From Theorem~\ref{NN:th31} ($s=0$)
we obtain the opposite inequality $F_v(2_r;r-1)\le r+7$.
\end{proof}

\begin{proof}[Proof of Theorem~\ref{NN:th15}(c)]
There is a 16-vertex graph $G$ such that $\ga(G)=3$ and $\cl=3$, because
$R(4,4)=18$ (see \cite{NN:6}). From $|V(G)|=16$ and $\ga(G)=3$ obviously
it follows that $\chi(G)\ge 6$. By~\eqref{NN:eq13}, $G\tov(2_5)$. So,
$G\in H_v(2_5;4)$. Hence $F_v(2_5;4)\le|V(G)|=16$.
\end{proof}

Theorem~\ref{NN:th15} is proved.

\begin{cor}\label{NN:cor41}
Let $G$ be a graph such that $f(G)\le 3$. Then $|V(G)|\ge\chi(G)+2f(G)$.
\end{cor}

\begin{proof}[\bf Proof.]
If $f(G)\le 2$, then Corollary~\ref{NN:cor41} follows from Lemma~\ref{NN:l21}.
Let $f(G)=3$ and $\chi(G)=r+1$. Then $\cl=r-2$. Since $\chi(G)\ne\cl$,
$\cl\ge 2$. Thus, $r\ge 4$. By~\eqref{NN:eq13}, $G\in H_v(2_r;r-1)$.
From Theorem~\ref{NN:th15}(a) we obtain $|V(G)|\ge r+7=\chi(G)+2f(G)$.
\end{proof}

\textbf{Remark.}
In $H_v(2_r;r-1)$, $r\ge 8$ there are more than one extremal graph.
For example, in $H_v(2_8;7)$ besides $K_2+P$ (see Theorem~\ref{NN:th31}),
the graph $C_5+C_5+C_5$ is extremal, too.

\section{Proof of Theorem~\ref{NN:th16}}

\begin{proof}[Proof of Theorem~\ref{NN:th16}(a)]
Let $G\in H_v(2_r;r-2)$. We need to prove that $|V(G)|\ge r+9$.
From Lemma~\ref{NN:l23} we have that
\[
|V(G)|\ge F_v(2_{r-1};r-2)+\ga(G).
\]
By Theorem~\ref{NN:th15}(a), $F_v(2_{r-1};r-2)\ge r+6$. Thus,
\beq{NN:eq51}
|V(G)|\ge r+6+\ga(G).
\eeq
We prove the inequality $|V(G)|\ge r+9$ by induction on $r$.
From Table~\ref{NN:t1} we see that
\beq{NN:eq52}
R(r-2,3)<r+8,
\quad
5\le r\le 7.
\eeq
Obviously, from $G\in H_v(2_r;r-2)$ it follows that $\chi(G)\ne\cl$.
Thus, $\ga(G)\ge 2$. From~\eqref{NN:eq51} we obtain $|V(G)|\ge r+8$.
This, together with~\eqref{NN:eq52}, implies $|V(G)|>R(r-2,3)$ if
$5\le r\le 7$. Since $\cl<r-2$, $\ga(G)\ge 3$.
By the inequality~\eqref{NN:eq51}, $|V(G)|\ge r+9$, $5\le r\le 7$.

Let $r\ge 8$. Obviously, it is enough to consider only the situation when
\beq{NN:eq53}
|V(G)|=F_v(2_r;r-2).
\eeq
By~\eqref{NN:eq53} and Lemma~\ref{NN:l22} we have that
\beq{NN:eq54}
\text{\itshape $G$ is a vertex-critical $(r+1)$-chromatic graph;}
\eeq
and
\beq{NN:eq55}
\cl=r-3.
\eeq
From~\eqref{NN:eq54} and~\eqref{NN:eq55} it follows that
\beq{NN:eq56}
f(G)=4.
\eeq
We shall consider two cases.

\textbf{Case 1.}
$|V(G)|<2r+1$. By~\eqref{NN:eq54} and Theorem~\ref{NN:th21}, we obtain that
\beq{NN:eq57}
G=G_1+G_2.
\eeq
From~\eqref{NN:eq57}, \eqref{NN:eq21} and~\eqref{NN:eq54},
obviously it follows that
\beq{NN:eq58}
\text{\itshape $G_i$, $i=1,2$ is a vertex-critical chromatic graph.}
\eeq
Let $f(G_1)=0$. Then, according to~\eqref{NN:eq58} $G_1$ is a complete graph.
Thus, it follows from~\eqref{NN:eq57} that $G=K_1+G'$. It is clear that
\[
G\in H_v(2_r;r-2)
\Rightarrow
G'\in H_v(2_{r-1};r-3).
\]
By the inductive hypothesis, $|V(G')|\ge r+8$. Hence, $|V(G)|\ge r+9$.
Let $f(G_i)\ne 0$, $i=1,2$. We see from~\eqref{NN:eq57}, \eqref{NN:eq23}
and~\eqref{NN:eq56} that $f(G_i)\le 3$, $i=1,2$. By Corollary~\ref{NN:cor41},
we have
\[
|V(G_i)|\ge\chi(G_i)+2f(G_i),
\quad
i=1,2.
\]
Summing these inequalities and using~\eqref{NN:eq21} and~\eqref{NN:eq23}
we obtain
\beq{NN:eq59}
|V(G)|\ge\chi(G)+2f(G).
\eeq
According to~\eqref{NN:eq54}, $\chi(G)=r+1$. Thus, from~\eqref{NN:eq59}
and~\eqref{NN:eq56} it follows that $|V(G)|\ge r+9$.

\textbf{Case 2.}
$|V(G)|\ge 2r+1$. Since $r\ge 8$, $2r+1\ge r+9$. Hence $|V(G)|\ge r+9$.
\end{proof}

\begin{proof}[Proof of Theorem~\ref{NN:th16}(b)]
By Theorem~\ref{NN:th16}(a), $F_v(2_r;r-2)\ge r+9$. Thus, we need to prove
the opposite inequality $F_v(2_r;r-2)\le r+9$ if $r\ge 8$. If $r\ge 9$,
this inequality follows from Theorem~\ref{NN:th31} ($s=1$). Let $r=8$.
By $R(6,3)=18$, \cite{NN:11} (see also \cite{NN:7}), there is a graph $Q$
such that $|V(Q)|=17$, $\ga(Q)=2$ and $\cl[Q]=5$. From $|V(Q)|=17$ and
$\ga(Q)=2$ obviously it follows that $\chi(Q)\ge 9$.
Thus, by~\eqref{NN:eq13}, $Q\tov(2_8)$. Hence $Q\in H_v(2_8;6)$ and
$F_v(2_8;6)\le|V(Q)|=17$. Theorem~\ref{NN:th16} is proved.
\end{proof}

\begin{cor}\label{NN:cor51}
Let $G$ be a graph such that $f(G)\le 4$. Then $|V(G)|\ge\chi(G)+2f(G)$.
\end{cor}

\begin{proof}[\bf Proof.]
If $f(G)\le 3$, then Corollary~\ref{NN:cor51} follows from
Corollary~\ref{NN:cor41}. Let $f(G)=4$ and $\chi(G)=r+1$.
Then $\cl=r-3$. Since $\chi(G)\ne\cl$, $\cl\ge 2$. Thus, $r\ge 5$.
By~\eqref{NN:eq13}, $G\in H_v(2_r;r-2)$. Using Theorem~\ref{NN:th16}(a),
we get $|V(G)|\ge r+9=\chi(G)+2f(G)$.
\end{proof}

Let $r\ge 3s+8$. Define
\[
\tilde Q=K_{r-3s-8}+Q+\underbrace{C_5+\dots+C_5}_s,
\]
where graph $Q$ is given in the proof of Theorem~\ref{NN:th16}(b). Since
$\cl[Q]=5$ and $\chi(Q)\ge 9$, we have by~\eqref{NN:eq21} and~\eqref{NN:eq22}
that $\cl[\tilde Q]=r-s-3$ and $\chi(\tilde Q)\ge r+1$.
According to~\eqref{NN:eq13}, $\tilde Q\in H_v(2_r;r-s-2)$.
Thus, $F_v(2_r;r-s-2)\le|V(\tilde Q)|$. Since $|V(\tilde Q)|=r+2s+9$,
we obtain the following

\begin{theorem}\label{NN:th52}
Let $r$ and $s$ be non-negative integers and $r\ge 3s+8$. Then
\[
F_v(2_r;r-s-2)\le r+2s+9.
\]
\end{theorem}

\begin{flushleft}
Faculty of Mathematics and Informatics\\
St.~Kl.~Ohridski University of Sofia\\
5, J.~Bourchier Blvd.\\
BG-1164 Sofia, Bulgaria\\
e-mail: \texttt{nenov@fmi.uni-sofia.bg}
\end{flushleft}

\end{document}